\documentclass[11pt,leqno]{article}
\usepackage[margin=1in]{geometry} 
\usepackage{amssymb,amsfonts,amsmath,bbm,mathrsfs,stmaryrd,mathtools,mathabx}
\usepackage{xcolor}
\usepackage{url}

\usepackage{extarrows}

\usepackage[shortlabels]{enumitem}
\usepackage{tensor}

\usepackage{xr}
\usepackage[T1]{fontenc}

\usepackage[colorlinks,
linkcolor=black!75!red,
citecolor=blue,
pdftitle={},
pdfproducer={pdfLaTeX},
pdfpagemode=None,
bookmarksopen=true,
bookmarksnumbered=true,
backref=page]{hyperref}

\usepackage{tikz}
\usetikzlibrary{arrows,calc,decorations.pathreplacing,decorations.markings,intersections,shapes.geometric,through,fit,shapes.symbols,positioning,decorations.pathmorphing}

\usepackage{braket}

\usepackage[amsmath,thmmarks,hyperref]{ntheorem}
\usepackage{cleveref}

\creflabelformat{enumi}{#2#1#3}

\crefname{section}{Section}{Sections}
\crefformat{section}{#2Section~#1#3} 
\Crefformat{section}{#2Section~#1#3} 

\crefname{subsection}{\S}{\S\S}
\AtBeginDocument{%
  \crefformat{subsection}{#2\S#1#3}%
  \Crefformat{subsection}{#2\S#1#3}%
}

\crefname{subsubsection}{\S}{\S\S}
\AtBeginDocument{%
  \crefformat{subsubsection}{#2\S#1#3}%
  \Crefformat{subsubsection}{#2\S#1#3}%
}

\theoremstyle{plain}

\newtheorem{lemma}{Lemma}[section]
\newtheorem{proposition}[lemma]{Proposition}
\newtheorem{corollary}[lemma]{Corollary}
\newtheorem{theorem}[lemma]{Theorem}

\theoremstyle{plain}
\theoremnumbering{Alph}
\newtheorem{theoremN}{Theorem}

\theoremstyle{plain}
\theorembodyfont{\upshape}
\theoremsymbol{\ensuremath{\blacklozenge}}

\newtheorem{definition}[lemma]{Definition}
\newtheorem{example}[lemma]{Example}

\newtheorem{remark}[lemma]{Remark}
\newtheorem{remarks}[lemma]{Remarks}

\crefname{definition}{definition}{definitions}
\crefformat{definition}{#2definition~#1#3} 
\Crefformat{definition}{#2Definition~#1#3} 

\crefname{ex}{example}{examples}
\crefformat{example}{#2example~#1#3} 
\Crefformat{example}{#2Example~#1#3} 

\crefname{exs}{example}{examples}
\crefformat{examples}{#2example~#1#3} 
\Crefformat{examples}{#2Example~#1#3} 

\crefname{remark}{remark}{remarks}
\crefformat{remark}{#2remark~#1#3} 
\Crefformat{remark}{#2Remark~#1#3} 

\crefname{remarks}{remark}{remarks}
\crefformat{remarks}{#2remark~#1#3} 
\Crefformat{remarks}{#2Remark~#1#3} 

\crefname{convention}{convention}{conventions}
\crefformat{convention}{#2convention~#1#3} 
\Crefformat{convention}{#2Convention~#1#3} 

\crefname{notation}{notation}{notations}
\crefformat{notation}{#2notation~#1#3} 
\Crefformat{notation}{#2Notation~#1#3} 

\crefname{table}{table}{tables}
\crefformat{table}{#2table~#1#3} 
\Crefformat{table}{#2Table~#1#3}

\crefname{lemma}{lemma}{lemmas}
\crefformat{lemma}{#2lemma~#1#3} 
\Crefformat{lemma}{#2Lemma~#1#3} 

\crefname{proposition}{proposition}{propositions}
\crefformat{proposition}{#2proposition~#1#3} 
\Crefformat{proposition}{#2Proposition~#1#3} 

\crefname{propositionN}{proposition}{propositions}
\crefformat{propositionN}{#2proposition~#1#3} 
\Crefformat{propositionN}{#2Proposition~#1#3} 

\crefname{corollary}{corollary}{corollaries}
\crefformat{corollary}{#2corollary~#1#3} 
\Crefformat{corollary}{#2Corollary~#1#3} 

\crefname{corollaryN}{corollary}{corollaries}
\crefformat{corollaryN}{#2corollary~#1#3} 
\Crefformat{corollaryN}{#2Corollary~#1#3} 

\crefname{theorem}{theorem}{theorems}
\crefformat{theorem}{#2theorem~#1#3} 
\Crefformat{theorem}{#2Theorem~#1#3} 

\crefname{theoremN}{theorem}{theorems}
\crefformat{theoremN}{#2theorem~#1#3} 
\Crefformat{theoremN}{#2Theorem~#1#3} 

\crefname{enumi}{}{}
\crefformat{enumi}{#2#1#3}
\Crefformat{enumi}{#2#1#3}

\crefname{assumption}{assumption}{Assumptions}
\crefformat{assumption}{#2assumption~#1#3} 
\Crefformat{assumption}{#2Assumption~#1#3} 

\crefname{construction}{construction}{Constructions}
\crefformat{construction}{#2construction~#1#3} 
\Crefformat{construction}{#2Construction~#1#3} 

\crefname{question}{question}{Questions}
\crefformat{question}{#2question~#1#3} 
\Crefformat{question}{#2Question~#1#3} 

\crefname{equation}{}{}
\crefformat{equation}{(#2#1#3)} 
\Crefformat{equation}{(#2#1#3)}

\numberwithin{equation}{section}

\theoremstyle{nonumberplain}
\theoremsymbol{\ensuremath{\blacksquare}}

\newtheorem{proof}{Proof}

\newcommand\bC{{\mathbb C}}

\newcommand\bP{{\mathbb P}}

\newcommand\bR{{\mathbb R}}
\newcommand\bS{{\mathbb S}}

\newcommand\bZ{{\mathbb Z}}

\newcommand\cA{{\mathcal A}}
\newcommand\cB{{\mathcal B}}

\newcommand\cE{{\mathcal E}}

\newcommand\cS{{\mathcal S}}
\newcommand\cT{{\mathcal T}}

\newcommand\fT{{\mathfrak T}}

\DeclareMathOperator{\id}{id}

\newcommand{\cat}[1]{\textsc{#1}}

\newcommand{\qedhere}{\mbox{}\hfill\ensuremath{\blacksquare}}

\renewcommand{\square}{\mathrel{\Box}}

\newcommand{\xrightarrowdbl}[2][]{%
  \xrightarrow[#1]{#2}\mathrel{\mkern-14mu}\rightarrow
}

\title{Finite-index phenomena and the topology of bundle singularities}
\author{Alexandru Chirvasitu}

\begin{document}

\date{}

\newcommand{\Addresses}{{
    \bigskip
    \footnotesize

    \textsc{Department of Mathematics, University at Buffalo}
    \par\nopagebreak
    \textsc{Buffalo, NY 14260-2900, USA}  
    \par\nopagebreak
    \textit{E-mail address}: \texttt{achirvas@buffalo.edu}


  }}

\maketitle

\begin{abstract}
  A classical branched cover is an open surjection of compact Hausdorff spaces with uniformly bounded finite fibers and analogously, a quantum branched cover is a unital $C^*$ embedding admitting a finite-index expectation. We show that whenever a compact Hausdorff space $Z$ contains a one-point compactification of an uncountable set, the incidence correspondence attached to the space of cardinality-$(\le n)$ subsets of $Z$ (for $n\ge 3$) is a classical branched cover that does not dualize to a quantum one. In particular, when $Z$ is dyadic, the resulting $C^*$ embeddings are quantum branched covers precisely when $Z$ is also metrizable. This provides a partial converse to an earlier result of the author's (to the effect that continuous, unital, subhomogeneous $C^*$ bundles over compact metrizable spaces are quantum branched) and settles negatively a question of Blanchard-Gogi\'{c}.

  There are also some positive results identifying classes of compact Hausdorff spaces (e.g. extremally disconnected or orderable) with the property that all (continuous, unital) $C^*$ bundles based thereon are quantum branched. 
\end{abstract}

\noindent {\em Key words:
  branched cover;
  conditional expectation;
  continuous selection;
  dyadic;
  finite index;
  metrizable;
  one-point compactification;
  subhomogeneous

}

\vspace{.5cm}

\noindent{MSC 2020:
  54C65; 
  46L30; 
  46M20; 
  46A04; 
  54C60; 
  54B15; 
  46L85; 
  26E25 

}


\section*{Introduction}

The present paper is a follow-up on \cite{2409.03531v1}, revisiting a number of problems on and around the {\it non-commutative branched coverings} of \cite[Definition 1.2]{pt_brnch} stemming from the material in \cite{bg_cx-exp}. Recall some of the pertinent terminology:
\begin{itemize}[wide]
\item A continuous surjection $Y\xrightarrowdbl{}X$ of compact Hausdorff spaces is a {\it branched cover} if it is open and has a finite upper bound on the cardinalities of the fibers. 

\item Building on this, a unital embedding $A\le B$ of $C^*$-algebras is a {\it non-commutative (or quantum) branched cover} if it admits an {\it expectation} $B\xrightarrow{E}A$ (i.e. \cite[Theorem II.6.10.2]{blk} a norm-1 idempotent map) of {\it finite index} in the sense of \cite[Definition 2]{fk_fin-ind}:
  \begin{equation*}
    \exists K\ge 1
    \quad:\quad
    KE-\id\ge 0,
  \end{equation*}
  i.e. $KE-\id$ is a {\it positive} map in the sense \cite[Definition II.6.9.1]{blk} that
  \begin{equation*}
    KE(b^*b)\ge b^*b,\quad\forall b\in B.
  \end{equation*}
  In alternative phrasing, the {\it K-constant}
  \begin{equation}\label{eq:k.const}
    K(E):=\inf\{K\ge 1\ |\ KE-\id\ge 0\}
    \quad
    \left(\text{see \cite[p.89]{fk_fin-ind}}\right)
  \end{equation}
  is finite. 
\end{itemize}
We occasionally refer to the former notion as a {\it classical} (or {\it ordinary}, or {\it plain}, etc.) branched cover, to emphasize the distinction. The quantum branched covers in this paper are always at least ``half-classical'', in that the base (small) $C^*$-algebra $A$ is commutative: $A\cong C(X)$ for some compact Hausdorff $X$. The language and perspective will at times switch freely between
\begin{itemize}[wide]
\item on the one hand, {\it Banach bundles} on $X$ \cite[Definition 1.1]{dg_ban-bdl}:
  \begin{equation*}
    \cE\xrightarrowdbl[\text{continuous open}]{\quad\pi\quad}X
    ,\quad
    \text{Banach space structures }
    \left(
      \text{{\it fiber} }\cE_x:=\pi^{-1}(x),\ \|\cdot\|_x
    \right),
  \end{equation*}
  with algebraic operations continuous in the obvious sense and
  \begin{equation}\label{eq:normfunc}
    \cE\ni p
    \xmapsto{\quad}
    \|p\|_{\pi(p)}
    \in \bR_{\ge 0}
  \end{equation}
  {\it upper semicontinuous} \cite[Chapter 3, Problem F]{kel-top} (the bundle is {\it continuous} or {\it (F)} if \Cref{eq:normfunc} is instead continuous);

\item on the other hand, {\it locally $C(X)$-convex Banach modules} \cite[\S 6.1]{hk_shv-bdl} over the function $C^*$-algebra $C(X)$. 
\end{itemize}
The back-and-forth passage is afforded by \cite[Scholium 6.7]{hk_shv-bdl}, and specializes in particular to the correspondence between continuous unital {\it $C^*$-algebra} bundles on $X$ \cite[p.9]{dg_ban-bdl} and unital {\it $C\left(X\right)$-algebras} $C\left(X\right)\to A$ in the sense of \cite[Definition 1.5]{zbMATH04056334} (unital $C^*$ morphisms taking values in the center $Z(A)$), continuous in that
\begin{equation*}
  \forall a\in A
  ,\quad
  X\ni x
  \xmapsto{\quad}
  \|a\|_x\in \bR_{\ge 0}
\end{equation*}
is continuous, with $\|\cdot\|_x$ denoting the norm on the $C^*$-algebra
\begin{equation*}
  A_x\text{ (the {\it fiber at $x$} of $A$)}:=A/I_x\cdot A
  ,\quad
  I_x:=\left\{f\in C\left(X\right)\ |\ f(x)=0\right\}.
\end{equation*}
The $C(X)$-algebra $A$ attached to a $C^*$-bundle $\cA\xrightarrowdbl{\pi}X$ is the {\it section space} \cite[p.9]{dg_ban-bdl}
\begin{equation*}
  \Gamma(\cA)
  :=
  \left\{X\xrightarrow[\text{continuous}]{s}\cA\ |\ \pi\circ s = \id_X\right\}.
\end{equation*}
Keeping with the spirit of moving freely between $C(X)$ algebras and $C^*$ bundles, we refer to one such, $\cA\xrightarrowdbl{}X$, as a (non-commutative) branched cover if the corresponding embedding $C_b(X)\lhook\joinrel\to \Gamma_b(\cA)$ is one in the sense of the previous item (`b' subscripts indicate bounded continuous functions and sections respectively). 

The bounded-fiber-size condition translates for Banach (here mostly $C^*$) bundles $\cA\xrightarrowdbl{}X$ to {\it subhomogeneity} \cite[p.3]{bg_cx-exp}: $\sup_x \dim \cA_x<\infty$. The vocabulary being in place, \cite[Problem 3.11]{bg_cx-exp}, motivating much of the sequel and slightly paraphrased, asks whether subhomogeneous continuous unital $C^*$ bundles with non-zero fibers over compact Hausdorff spaces are quantum branched covers. There are various strands of evidence: 
\begin{itemize}[wide]
\item On the one hand, \cite[Theorem 1.1, implications (1) $\Rightarrow$ (2) and/or (3)]{pt_brnch} asserts an affirmative answer in the fully classical case of commutative $\Gamma(\cA)\cong C(Y)$.

\item On the other, \cite[Theorem 2.12]{2409.03531v1} answers the question affirmatively in the form given here for {\it metrizable} (compact Hausdorff) base spaces $X$. 
\end{itemize}

We will see below that the answer cannot be `yes' in the phrased generality even classically; {\it some} constraints on the base space $X$ (or alternatively, the bundle itself) are needed: \Cref{th:pt.non.metr,th:char.metr} identify classes of examples where metrizability is in fact equivalent to the existence of the requisite finite-index expectation. Compressing and streamlining the results:

\begin{theoremN}\label{thn:metr.needed.dyad}
  For compact Hausdorff $Z$ and a positive integer $n\ge 3$ set
  \begin{equation*}
    \begin{aligned}
      X&:=Z_{[n]}:=\left\{{\bf z}\subseteq Z\ :\ 1\le |{\bf z}|\le n\right\}\text{ with the {\it Vietoris topology} \cite[Definition 1.3.3]{kt_corresp}}\\
      Y&:=\left\{(z,x)\in Z\times X\ :\ z\in x\right\}\\
      Y\xrightarrowdbl{\pi}X
       &:=\text{the second projection}. 
    \end{aligned}    
  \end{equation*}

  \begin{enumerate}[(1),wide]
  \item If $Z$ contains the {\it one-point compactification} \cite[Definition 19.2]{wil_top} of a discrete uncountable space then the embedding $C(X)\lhook\joinrel\xrightarrow{\pi^*}C(Y)$ dual to the classical branched cover $\pi$ is not a quantum branched cover. 

  \item If $Z$ is {\it dyadic}, i.e. \cite[Problem 3.12.12]{eng_top_1989} a continuous image of $(\bZ/2)^{\alpha}$ for some cardinal $\alpha$, $\pi^*$ is a quantum branched cover if and only if $X$ (or equivalently, $Z$) is metrizable.  \qedhere
  \end{enumerate}
\end{theoremN}

This is all reminiscent of and very much in line with results (\cite[Theorem 5.2]{rs_cont-sel}, citing \cite[Theorem 2]{zbMATH03569401}) to the effect that metrizability is in some fashion necessary in Michael-type {\it selection theorems} such as the paradigmatic \cite[Theorem 3.2'']{mich_contsel-1}. For the class of spaces $X$ singled out by the theorem, the fiber cardinalities $|\pi^{-1}(x)|$ vary ``wildly enough'' to preclude the existence of a finite-index expectation $C(Y)\xrightarrow{E}C(X)$. 

For a plain branched cover $Y\xrightarrowdbl{\pi}X$, the construction of the desired $E$ proposed in \cite[Theorem 4.3]{pt_brnch} (or rather implicit there, since that proof is rather concerned with the associated {\it Hilbert-module inner product} \cite[Definition 15.1.1]{wo} $\braket{f\mid g}:=E(f^*g)$) proceeds recursively along the {\it strata}
\begin{equation*}
  X_d
  :=
  \{x\in X\ :\ |\pi^{-1}(x)|=d\}
  \subseteq X.
\end{equation*}
Even when such expectations {\it do} exist (so under more pleasant circumstances than those of \Cref{thn:metr.needed.dyad}), \Cref{ex:pt_2-3_sheets,ex:3.sheeted.disk,ex:prod.2.uncount.ord} below indicate various ways in which such a recursive procedure might not go through. 

In light of the above remark on wild fiber oscillation in \Cref{thn:metr.needed.dyad}, \Cref{th:tame.strat} classes of $C^*$ bundles where such oscillation does not obtain and which are consequently quantum branched covers. The fact that, regardless of metrizability, {\it all} reasonable bundles over {\it orderable} compact Hausdorff spaces are non-commutative branched covers then follows (\Cref{cor:ord.nc.brnch}). 

In conformity with the commutative $C^*$ version in \cite[post Lemma 2.3]{pt_brnch}, define the {\it strata} associated to a Banach bundle $\cA\xrightarrowdbl{}X$ by
\begin{equation}\label{eq:strata}
  \begin{aligned}
    X_{d}=X_{\cA=d}
    &:=
      \left\{x\in X\ |\ \dim \cA_x=d\right\}\\
    X_{\le d}=X_{\cA\le d}
    &:=
      \left\{x\in X\ |\ \dim \cA_x\le d\right\}
      =\bigcup_{d'\le d}X_{d'}.\\
  \end{aligned}    
\end{equation}

\begin{theoremN}\label{thn:well.stratified}
  Let $\cA\xrightarrowdbl{\pi}X$ be a unital continuous subhomogeneous $C^*$ bundle with non-zero fibers over a paracompact base space. 
  \begin{enumerate}[(1),wide]
  \item If for every $d\in \bZ_{\ge 1}$ the top stratum $X_d$ in $X_{\le d}$ admits a partition into clopen subspaces $U\subseteq X_{\le d}$ such that
    \begin{equation*}
      \forall e<d
      ,\quad
      X_e\cap \overline{U}
      \text{ is clopen in }X_{<d}\cap\overline{U}
    \end{equation*}
    then the bundle is a quantum branched cover. 

  \item In particular, if $X$ is compact Hausdorff carrying the {\it order topology} \cite[\S 14]{mnk} of a total ordering then all bundles as in the statement are quantum branched covers.  \qedhere
  \end{enumerate}
\end{theoremN}

\subsection*{Acknowledgements}

I am grateful for kind interest and useful remarks from I. Gogi{\'c} and E. Troitsky. 



\section{Positive and no-go results on quantum branched covers}\label{se:yy}


The portion of \cite[Theorem 1.1]{pt_brnch} relegated to \cite[Theorem 4.3]{pt_brnch} states that $C^*$ morphisms $C\left(X\right)\lhook\joinrel\to C\left(Y\right)$ dual to (plain, topological) branched covers $Y\xrightarrowdbl{} X$ are non-commutative branched covers. \Cref{th:pt.non.metr} argues that this cannot be so without some additional assumptions (e.g. metrizability for the spaces in question).

For spaces $Z$ and positive integers $n$ we write $Z_{[n]}$ for the corresponding {\it $n^{th}$ truncated power-set} (what \cite[p.608]{zbMATH03569401} and \cite[p.164]{MR450972} would denote by $Z(n)$):
\begin{equation}\label{eq:zn2z_n}
  Z^n
  \ni
  (z_i)_{i=1}^n
  \xrightarrowdbl[]{\quad\text{quotient map }\quad}
  \left\{z_i\right\}_{i=1}^n
  \in
  Z_{[n]}
  :=
  \left\{{\bf z}\subseteq Z\ :\ 1\le |{\bf z}|\le n\right\}.
\end{equation}
$Z_{[n]}$ thus consists of the non-empty $(\le n)$-element subsets of $Z$ (so is empty when $Z$ is), and carries the quotient topology inherited from \Cref{eq:zn2z_n}. Equivalently, this is the the {\it Vietoris topology} \cite[\S 1.2]{ct_vietoris} on the collection of closed subsets of $Z$. Define the {\it incidence correspondence} (e.g. \cite[Example 6.12]{har_ag}) associated to that collection of subsets:
\begin{equation*}
  Z_{[n]}^{\subseteq}
  :=
  \left\{(z,x)\in Z\times Z_{[n]}\ |\ z\in x\right\}
  \subseteq
  Z\times Z_{[n]}.
\end{equation*}

\begin{remark}
  \cite[Paragraph preceding Theorem B.1]{MR450972} refers to $Z_{[n]}$ as the {\it symmetric product} of $Z$, but that terminology is in conflict with, say, the algebraic geometry literature, where symmetric products (or powers) are pervasive. There (e.g. \cite[p.18]{acgh1}, \cite[Example 10.23]{har_ag}, \cite[\S 1]{Polizzi}, \cite{CaCi93}), the phrase {\it symmetric product} rather refers to
  \begin{equation}\label{eq:zn2z^n}
    Z^n
    \xrightarrow{\quad\text{quotient map}\quad}
    Z^{[n]}
    :=
    Z^n/\left(\text{permutation action by $S_n$}\right).
  \end{equation}
  This is thus a larger quotient of $Z^n$: \Cref{eq:zn2z_n} factors through \Cref{eq:zn2z^n}.

  The difference between $Z_{[n]}$ and $Z^{[n]}$ is (respectively) that between sets and {\it multi}sets: in the latter space multiplicity information is retained, so that $\{z,z',z'\}$ and $\{z,z,z'\}$ count as distinct objects for $z\ne z'$; in the former, both are identified with the underlying plain set $\{z,z'\}$. 
\end{remark}

\Cref{le:inc.corr.brnch} below is a simple observation to build on. We remind the reader (e.g. \cite[p.712]{zbMATH06329568} or \cite[Definition 7.1.1]{kt_corresp}) that a map $X\xrightarrow{\varphi} 2^Y$ for topological $X$ and $Y$ is
\begin{itemize}[wide]
\item {\it lower semicontinuous (LSC)} if
  \begin{equation*}
    \left(\forall \text{ open }W\subseteq Y\right)
    \left(\forall x\in X\right)
    \quad:\quad
    \left\{x'\in X\ |\ \varphi(x')\cap W\ne\emptyset\right\}
    \text{ is open in }X;
  \end{equation*}

\item {\it upper semicontinuous (USC)} if
  \begin{equation*}
    \left(\forall \text{ open }W\subseteq Y\right)
    \left(\forall x\in X\right)
    \quad:\quad
    \left\{x'\in X\ |\ \varphi(x')\subseteq W\right\}
    \text{ is open in }X.
  \end{equation*}
\end{itemize}

\begin{lemma}\label{le:inc.corr.brnch}
  For compact Hausdorff $Z$ and $n\in \bZ_{>0}$ the obvious surjection
  \begin{equation}\label{eq:zn.corresp}
    Z_{[n]}^{\subseteq}
    \lhook\joinrel\xrightarrow{\quad}
    Z\times Z_{[n]}
    \xrightarrowdbl{\quad\text{$2^{nd}$ projection}\quad}
    Z_{[n]}
  \end{equation}
  is a branched cover. 
\end{lemma}
\begin{proof}
  It is of course a continuous surjection, and the fiber cardinalities are all $\le n$.

  Given that the elements of $Z_{[n]}\subseteq 2^Z$ are all closed, the closure of the left-hand embedding in \Cref{le:inc.corr.brnch} is equivalent \cite[Theorems 7.1.15 and 7.1.16]{kt_corresp} to (and hence follows from) the upper semicontinuity of the embedding $Z_{[n]}\subseteq 2^Z$; the domain of \Cref{eq:zn.corresp} is thus compact Hausdorff.

  Finally, the openness of the map follows in like fashion from the {\it lower} semicontinuity of the same inclusion $Z_{[n]}\subseteq 2^Z$: the general principle at work (and a simple exercise) is that for topological $X$ and $Y$ and an LSC map $X\xrightarrow{\varphi} 2^Y$, the composition
  \begin{equation*}
    \bigg(\text{{\it graph} \cite[Definition 6.1.2]{kt_corresp} of $\varphi$}\bigg)
    :=
    \left\{(y,x)\in Y\times X\ |\ y\in \varphi(x)\right\}
    \xrightarrow{\quad\text{$2^{nd}$ projection}\quad}
    X
  \end{equation*}
  is open.
\end{proof}


\begin{theorem}\label{th:pt.non.metr}
  Let $Z$ be a compact Hausdorff space containing the {\it Alexandroff compactification} \cite[Definition 19.2]{wil_top} of a discrete uncountable space and $n\in \bZ_{\ge 3}$.
  
  The embedding $C\left(X\right)\lhook\joinrel\xrightarrow{\iota} C\left(Y\right)$ dual to
  \begin{equation*}
    Y:=
    Z_{[n]}^{\subseteq}
    \xrightarrowdbl{\quad\text{$\pi:=$\Cref{eq:zn.corresp}}\quad}
    Z_{[n]}
    =:X
  \end{equation*}
  does not admit a conditional expectation, so in particular is not a non-commutative branched cover in the sense of \cite[Definition 1.2]{pt_brnch}.
\end{theorem}
\begin{proof}
  The map
  \begin{equation*}
    X\ni {\bf z}=\left\{z_1,z_2,z_3\right\}
    \xmapsto{\quad}
    C_{{\bf z}}
    :=
    \text{convex hull }\mathrm{cvx}\{z_i\}
    \cong
    \text{state space }\cS(\pi^{-1}({\bf z}))
  \end{equation*}
  is LSC for the weak$^*$ topology on the dual space $C\left(Z\right)^*$, and an expectation $C\left(Y\right)\xrightarrow{E}C\left(X\right)$ would effect a {\it continuous selection} \cite[pp.711-712]{zbMATH06329568}
  \begin{equation}\label{eq:cont.sel.z3}
    X\ni x={\bf z}
    \xmapsto{\quad}
    E_{{\bf z}}
    \in
    C_{{\bf z}}
  \end{equation}
  subordinate to it. That such a continuous selection cannot exist under for
  \begin{equation*}
    Z=\text{one-point compactification }\Delta^+\text{ of $\Delta$}
    ,\quad
    |\Delta|>\aleph_0
  \end{equation*}
  and $n=3$ follows from \cite[Theorem B.5.2]{rs_cont-sel} (which in turn rephrases a combination of \cite[Theorem 2]{zbMATH03569401} and its proof), applied to the space $Z$ embedded homeomorphically in the weak$^*$-topologized $C\left(Z\right)^*$ as the space of characters of the $C^*$-algebra $C\left(Z\right)$. To conclude, simply note that a selection for $Z\supseteq \Delta^+$ and $n\ge 3$ would entail one for $Z=\Delta^+$ and $n=3$ by restriction, the square
  \begin{equation*}
    \begin{tikzpicture}[>=stealth,auto,baseline=(current  bounding  box.center)]
      \path[anchor=base] 
      (0,0) node (l) {$(\Delta^+)^{\subseteq}_{[3]}$}
      +(2,.5) node (u) {$Z_{[n]}^{\subseteq}$}
      +(2,-.5) node (d) {$(\Delta^+)_{[3]}$}
      +(4,0) node (r) {$Z_{[n]}$}
      ;
      \draw[right hook->] (l) to[bend left=6] node[pos=.5,auto] {$\scriptstyle $} (u);
      \draw[->>] (u) to[bend left=6] node[pos=.5,auto] {$\scriptstyle $} (r);
      \draw[->>] (l) to[bend right=6] node[pos=.5,auto,swap] {$\scriptstyle $} (d);
      \draw[right hook->] (d) to[bend right=6] node[pos=.5,auto,swap] {$\scriptstyle $} (r);
    \end{tikzpicture}
  \end{equation*}
  being a {\it pullback} \cite[Definition 11.8]{ahs} in the category of compact Hausdorff spaces.
\end{proof}


\begin{remarks}\label{res:neg.res}
  \Cref{th:pt.non.metr} answers a few questions posed in the literature.
  
  \begin{enumerate}[(1),wide]  

  \item In \cite[discussion following Remarque on p.155]{blnch}, having argued for the existence of appropriately faithful expectations onto $C\left(X\right)$ definable on continuous {\it separable } $C\left(X\right)$-algebras $C\left(X\right)\to A$ in the sense of \cite[Definition 1.5]{zbMATH04056334}, Blanchard observes that it would be interesting to give a counterexample when the fibers $A_x$ are non-separable. The examples provided by \Cref{th:pt.non.metr} do slightly more: the fibers are in that case finite-dimensional and abelian (1, 2 or 3-dimensional), and it is rather the global algebra $A=C\left(Y\right)$ itself that is not separable.

  \item In the same spirit there is \cite[Problem 3.11]{bg_cx-exp}, answered by \Cref{th:pt.non.metr} negatively in the generality in which that question was asked: for $Z$, $X:=Z_{[n]}$ and $Y:=Z_{[n]}^{\subseteq}$ as in the statement, $C\left(X\right)\lhook\joinrel\to C\left(Y\right)$ is a subhomogeneous $C\left(X\right)$-algebra. Nevertheless, there are no finite-index conditional expectations (for indeed, there are no expectations at all).  
  \end{enumerate}
\end{remarks}

We pause briefly to remark that yet another way in which the examples provided by \Cref{th:pt.non.metr} are ill-behaved. Note that for continuous bundles (as virtually all are, in the present paper) the strata $X_{\le d}$ of \Cref{eq:strata} are closed \cite[Theorem 18.3]{gierz_bdls} and hence $X_d$ are {\it locally} closed (i.e. \cite[\S I.3.3]{bourb_top_en_1} open in their respective closures). The ``bottom'' stratum $X_d$, $d=\min_x\dim \cA_x$ of a continuous subhomogeneous $C^*$ bundle is always closed, but in the class of examples just mentioned all higher strata fail to even be {\it normal}, i.e. \cite[Definition 15.1]{wil_top} disjoint closed subsets thereof will not always have disjoint neighborhoods. In particular \cite[Theorem 20.10]{wil_top}, those strata will also fail to be {\it paracompact} \cite[Definition 20.6]{wil_top}.

\begin{lemma}\label{le:strt.not.norm}
  Let $Z$ be a compact Hausdorff space as in \Cref{th:pt.non.metr} and $n\in \bZ_{>0}$.
  
  None of the strata $\left(Z_{[n]}\right)_d$, $2\le d\le n$ associated to the branched cover \Cref{eq:zn.corresp} are normal.
\end{lemma}
\begin{proof}
  Said strata $\left(Z_{[n]}\right)_d$ are the respective collections of subsets of precise cardinality $d$:
  \begin{equation*}
    \left(Z_{[n]}\right)_d
    =
    Z_{(d)}
    :=
    \left\{{\bf z}\subseteq Z\ :\ |{\bf z}|=d\right\}.
  \end{equation*}
  A closed embedding
  \begin{equation*}
    \Delta^+:=\text{one-point compactification $\Delta\sqcup\{\infty\}$ of discrete }\Delta
    \subseteq Z
  \end{equation*}
  induces one between the corresponding collections of $d$-element subsets; because \cite[Theorem 15.4(a)]{wil_top} normality transports over to closed subspaces, it will be enough to assume $Z=\Delta^+$ for discrete uncountable $\Delta$. To simplify matters further, note that any fixed $(d-2)$-element subset $D\subset \Delta$ gives a closed embedding
  \begin{equation*}
    \left(\Delta\setminus D\right)^+_{(2)}
    \ni
    {\bf z}
    \xmapsto{\quad}
    \left({\bf z}\sqcup D\right)
    \in
    \Delta^+_{(d)};
  \end{equation*}
  substituting $\Delta\setminus D$ for $\Delta$ and $2$ for $d$, it will thus suffice to show that the set $\Delta^+_{(2)}$ of distinct pairs in $\Delta^+$ is not normal.

  More specifically, partition
  \begin{equation*}
    \Delta=\Delta_1\sqcup \Delta_2
    ,\quad
    |\Delta_i|=|\Delta|
  \end{equation*}
  so that $\{(\infty,\delta)\ |\ \delta\in \Delta_i\}$ are closed and disjoint. I claim that they do not have disjoint neighborhoods; even more precisely, every neighborhood of one clusters at some point of the other. To see this, note first that a neighborhood of $(\infty,\delta)$, $\delta\in\Delta$ will contain a set of the form
  \begin{equation}\label{eq:ddel2}
    D_{\delta\mid 2}
    :=
    \left\{\{\infty,\delta\}\cup\bigcup_{\delta'\in D_{\delta}}\left\{\delta',\delta\right\}\right\}
    ,\quad
    D_{\delta}\underset{\text{cofinite}}{\subseteq}\Delta.
  \end{equation}
  Suppose we prove that
  \begin{equation}\label{eq:some.del1.inf}
    \exists \delta_1\in \Delta_1
    \quad:\quad
    \sharp\left\{\delta\in \Delta_2\ :\ \{\delta_1,\delta\}\in D_{\delta}\right\}\ge \aleph_0.
  \end{equation}
  Every neighborhood of $(\infty,\Delta_1)$ (hence also of $\{(\infty,\delta)\ |\ \delta\in \Delta_i\}$) will then intersect the neighborhood $\bigcup_{\delta\in \Delta_2}D_{\delta\mid 2}$ of $(\infty,\Delta_2)$, concluding the proof.

  It thus remains to settle \Cref{eq:some.del1.inf}. To that end, note that the sets $D_{\delta}\cap \Delta_1$ with $\delta\in \Delta_2$ are cofinite subsets of the uncountable set $\Delta_1$, indexed over a set of the same cardinality. Relabeling both the indexing set $\Delta_2$ and the main set of interest $\Delta_1$ by $\Delta$, the claim \Cref{eq:some.del1.inf} follows from the following more general remark by specialization at $\kappa:=\aleph_0$: if $\kappa$ is an infinite cardinal, then
  \begin{equation}\label{eq:many.large.sets}
    \left.
      \begin{aligned}
        |\Delta|
        &\ge \kappa^+\text{  ({\it cardinal successor} \cite[p.29]{jech_st})}\\
        D_{\delta}
        &\subseteq \Delta
          ,\quad
          |\Delta\setminus D_{\delta}|<\kappa
          ,\quad
          \delta\in \Delta
      \end{aligned}
    \right\rbrace
    \xRightarrow{\quad}
    \exists\delta_0
    \ :\
    \sharp\left\{\delta\ :\ D_{\delta}\ni \delta_0\right\}\ge \kappa.
  \end{equation}
  Indeed, if
  \begin{equation*}
    \sharp
    \left(
      D_i:=\left\{\delta\ :\ D_{\delta}\ni \delta_i\right\}
    \right)
    <\kappa    
    ,\quad i\in \kappa
  \end{equation*}
  then $D_{\delta}$ would miss all $\delta_i$, $i\in \kappa$ whenever $\delta$ belongs to (the $\kappa^+$-sized set) $\Delta\setminus\bigcup_i D_i$.
\end{proof}

\begin{remarks}\label{res:same.comb.princ}
  \begin{enumerate}[(1),wide]
    
  \item The spaces $Z_{(d)}$ of \Cref{le:strt.not.norm}, while not normal, satisfy the ``next best'' separation axiom of {\it complete regularity} \cite[Definition 14.8]{wil_top}: points $x$ are separated from closed subsets $F\not\ni x$ by functions, in the sense that there is a globally-defined continuous function vanishing at $x$ and taking the value 1 over $F$. Indeed, $Z_{(d)}$ are by definition subspaces of the compact Hausdorff $Z_{[n]}$ and complete regularity is inherited by subspaces \cite[Theorem 14.10(a)]{wil_top}. 

  \item There are some family resemblances between the completely regular non-normal spaces $\Delta^+_{(2)}$ and other examples in the literature featuring this combination of separation properties:
    \begin{itemize}[wide]
    \item the {\it Niemytzki tangent disk topology} of  \cite[Example 82]{ss_countertop}: the upper half-plane $\{(x,y)\ :\ y\ge 0\}$ with the usual topology on the interior $\{y>0\}$ and open neighborhoods
      \begin{equation*}
        \{(x,0)\}
        \sqcup
        \bigg(\text{interior of some disk tangent to $\{y=0\}$ at $(x,0)$}\bigg). 
      \end{equation*}

    \item Similarly, the subspace 
      \begin{equation*}
        \{y=0\}\cup\left\{\left(\frac {m}{n^2},\ \frac 1{n}\right)\ :\ m\in \bZ,\ n\in \bZ_{>0}\right\}
      \end{equation*}
      of $\bR^2$ exhibited in \cite{MR20769}, discrete away from the $x$-axis and with local neighborhoods around $(x,0)$ consisting of triangles
      \begin{equation*}
        \left\{(u,v)\ :\ 0\le v\le\frac 1n,\ |u-x|\le v\right\}.
      \end{equation*}
    \end{itemize}
    In all cases the space in question contains an exceptional locus (the $x$-axis in the latter two examples, $\{\{\infty,\delta\}\in \Delta^+_{(2)}\}$ in the former) whose points have neighborhoods with ``large overlaps''.     
    
    
  \item Having cited \cite{zbMATH03569401,MR450972} in connection to \Cref{th:pt.non.metr}, we note that \cite[Lemma 1]{zbMATH03569401} (also \cite[proof of Theorem B.1, part 1]{MR450972}) is driven by the same combinatorial principle \Cref{eq:many.large.sets} employed in the proof of \Cref{le:strt.not.norm}. That result (in its \cite[Lemma 1]{zbMATH03569401} phrasing) states that given uncountable $\Delta$ and a map
    \begin{equation*}
      \Delta_{[2]}\xrightarrow{\quad\varphi\quad}\Delta
      \quad\text{with}\quad
      \varphi(D)\in D
      ,\quad
      \forall D\in \Delta_{[2]}
    \end{equation*}
    there is some $\delta_0$ for which $\varphi(\left\{\delta,\delta_0\right\})=\delta_0$ for infinitely many $\delta$. This is a consequence of \Cref{eq:many.large.sets} (with $\kappa:=\aleph_0$): assuming the contrary, set
    \begin{equation*}
      D_{\delta}
      :=
      \left\{\delta'\in \Delta\ :\ \varphi\left(\left\{\delta,\delta'\right\}\right)=\delta'\right\},
    \end{equation*}
    The assumed cofiniteness of the $D_{\delta}$ plugs this into \Cref{eq:many.large.sets}, whence a contradiction.

  \item The same \Cref{eq:many.large.sets} is also, ultimately, what the proof of \cite[Proposition]{zbMATH03569401} relies on. Fleshing out slightly the abbreviated sketch given in \cite[Theorem 5.2]{rs_cont-sel}: assume a selection \Cref{eq:cont.sel.z3} for $X=\Delta^+_{[3]}$, where as usual, $\Delta^+=\Delta\sqcup\{\infty\}$ is the one-point compactification of an uncountable discrete $\Delta$. For each $\delta\in \Delta$, the set
    \begin{equation*}
      D_{\delta;\varepsilon}
      :=
      \left\{
        \delta'\in \Delta
        \ :\
        E_{\left\{\delta,\delta',\infty\right\}}
        =
        c\delta+c'\delta'+c_{\infty}\infty
        ,\
        \left\|(c,c',c_{\infty})-\left(\frac 12,\frac 14,\frac 14\right)<\varepsilon\right\|
      \right\}
    \end{equation*}
    is first proven cofinite, so that (by \Cref{eq:many.large.sets}) some $\delta'$ is contained in infinitely many $D_{\delta;\varepsilon}$. But then, assuming $\varepsilon>0$ sufficiently small, $D_{\delta';\varepsilon}$ will omit those infinitely many $\delta$; this contradicts the aforementioned cofiniteness. 
  \end{enumerate}
  In summary: the phenomenon underlying the negative result of \Cref{th:pt.non.metr} also manifests as lack of normality (and hence paracompactness) for the strata of the same (type of) bundle. 
\end{remarks}

\Cref{th:pt.non.metr} applies of course to one-point compactifications $\Delta^+$ of uncountable discrete $\Delta$, but also, as noted in the proof of \cite[Theorem 2]{zbMATH03569401} (citing \cite[p. 499, Proof of Theorem4]{zbMATH03209749}), to non-metrizable {\it dyadic} spaces \cite[Definition A4.30]{hm5}: (compact Hausdorff) continuous images of {\it Cantor cubes} $\{0,1\}^{\aleph}$ for some cardinal $\aleph$. In fact, we can leverage \Cref{th:pt.non.metr} into a characterization of metrizability for such spaces (one among the many others known).

\begin{theorem}\label{th:char.metr}
  For a dyadic compact Hausdorff space $Z$ the following conditions are equivalent.

  \begin{enumerate}[(a),wide]
  \item\label{item:th:char.metr:metr} $Z$ is metrizable.

  \item\label{item:th:char.metr:metr.sympow} All (some) truncated power-sets $Z_{[n]}$ are metrizable. 

  \item\label{item:th:char.metr:gen.br} For every (some) $n\in \bZ_{\ge 3}$ every subhomogeneous $C\left(Z_{[n]}\right)$-algebra is a non-commutative branched cover.

  \item\label{item:th:char.metr:gen.exp} For every (some) $n\in \bZ_{\ge 3}$ every subhomogeneous $C\left(Z_{[n]}\right)$-algebra admits a conditional expectation.

  \item\label{item:th:char.metr:part.br} For every (some) $n\in \bZ_{\ge 3}$ the embedding $C\left(Z_{[n]}\right)\le C\left(Z_{[n]}^{\subseteq}\right)$ dual to \Cref{eq:zn.corresp} is a non-commutative branched cover.

  \item\label{item:th:char.metr:part.exp} For every (some) $n\in \bZ_{\ge 3}$ the embedding $C\left(Z_{[n]}\right)\le C\left(Z_{[n]}^{\subseteq}\right)$ admits an expectation.
  \end{enumerate}
\end{theorem}
\begin{proof}
  That \Cref{item:th:char.metr:metr} $\Leftrightarrow$ \Cref{item:th:char.metr:metr.sympow} is obvious: $Z$ is a subspace of every $Z_{[n]}$ and conversely, $Z_{[n]}$ is a quotient of $Z^n$ (metrizability transports along both embeddings and quotient maps between compact Hausdorff spaces).

  \Cref{item:th:char.metr:metr.sympow} $\Rightarrow$ \Cref{item:th:char.metr:gen.br} is a consequence of \cite[Theorem 2.12]{2409.03531v1}, and the `every' branch of \Cref{item:th:char.metr:gen.br} obviously implies \Cref{item:th:char.metr:gen.exp,item:th:char.metr:part.br,item:th:char.metr:part.exp} in all of their versions.

  Finally, \Cref{th:pt.non.metr} proves that the weaker version of \Cref{item:th:char.metr:part.exp} implies \Cref{item:th:char.metr:metr}, for non-metrizable dyadic spaces contain one-point compactifications of uncountable discrete spaces (as noted in \cite[paragraph preceding Remark 2]{zbMATH03569401}, citing \cite[p.499, Proof of Theorem 4]{zbMATH03209749}).
\end{proof}

\begin{example}\label{ex:lex.sq}
  Let $Z$ be the unit square $[0,1]^2$ topologized as in \cite[Example 48]{ss_countertop}, with the {\it order topology} \cite[Example 39]{ss_countertop}:
  \begin{equation*}
    (a,b)<(c,d)
    \iff
    a<c
    \quad\text{or}\quad
    a=c,\ b<d.
  \end{equation*}
  The embedding $C\left(Z_{[n]}\right)\le C\left(Z^{\subseteq}_{[n]}\right)$ not only admits a (finite-index) conditional expectation, but in fact splits in the category of $C^*$-algebras (i.e. has a left inverse): $Z$ is totally ordered, so the surjection \Cref{eq:zn.corresp} admits the splitting which assigns to each $(\le n)$-element subset of $Z$ its smallest (say) element. $Z$ is not metrizable though, as observed in \cite[Example 48, item 4.]{ss_countertop}: it is compact Hausdorff but not separable.
\end{example}

\begin{remarks}\label{res:ordrbl}
  \begin{enumerate}[(1),wide]
  \item\label{item:res:ordrbl:lex.not.dy} It is clear from \Cref{th:char.metr} and \Cref{ex:lex.sq} that the latter's lexicographically-ordered and topologized square $Z:=[0,1]^2$ is not dyadic, but one of course need not rely on the theorem for this: checking that $Z$ contains no one-point compactifications of discrete uncountable spaces is an easy matter (an uncountable subset will by necessity have more than one {\it cluster point} \cite[Definition 4.9]{wil_top}), so, being non-metrizable, $Z$ also cannot be dyadic by the already-mentioned result \cite[Proof of Theorem 4]{zbMATH03209749} of Efimov's.

    Alternatively: $Z$ is {\it first countable} by \cite[Example 48 4.]{ss_countertop} (in he sense \cite[p.7]{ss_countertop} that points have countable neighborhood bases), whence it would also be {\it second} countable (have a countable basis) and hence \cite[Theorem 34.1]{mnk} metrizable if it were dyadic \cite[Theorem 4]{zbMATH03222595}.

    Or again: Cantor cubes satisfy the {\it countable chain condition} (\cite[Chapter 5, Problem O(f)]{kel-top}: collections of mutually-disjoint non-empty open subsets are at most countable), whereas $Z$ does not.

  \item\label{item:res:ordrbl:mich.ord} The splitting of the $C^*$ morphism $C\left(Z_{[n]}\right)\le C\left(Z^{\subseteq}_{[n]}\right)$ of \Cref{ex:lex.sq} is a particular instance of \cite[Theorem 1.9]{zbMATH03067268}, noting that for Hausdorff spaces $X$ totally orderable in such a manner that (a) the order topology is at most as fine as the original topology and (b) closed sets have smallest elements there is a continuous selection
    \begin{equation*}
      \left(\text{compact subsets of $Z$},\ \text{Vietoris}\right)
      \ni K
      \xmapsto{\quad}
      \left(z_K\in K\right)
      \in Z;
    \end{equation*}
    that, in turn, provides a splitting
    \begin{equation*}
      Z_{[n]}\ni {\bf z}
      \xmapsto{\quad}
      (z_{\bf z},\ {\bf z})
      \in
      Z_{[n]}^{\subseteq}
    \end{equation*}
    of \Cref{eq:zn.corresp}. The main feature in the example, in short, is the total orderability of the topology.
  \end{enumerate}
\end{remarks}

By way of supplementing \Cref{th:pt.non.metr}, \Cref{ex:pt_2-3_sheets,ex:3.sheeted.disk} below attempt to identify some ways in which the {\it proof} of \cite[Theorem 4.3]{pt_brnch} does not function as claimed. In both cases, however, its {\it conclusion} is not difficult to prove valid (and indeed, in both examples the base space $X$ is metrizable and hence the conclusion must \cite[Theorem 2.12]{2409.03531v1} hold).

\begin{example}\label{ex:pt_2-3_sheets}
  We describe a branched cover $Y\to X$ in the sense of \cite[Definition 1.2]{pt_brnch}, ignoring any compactness concerns; it will be obvious that one can always restrict the map $Y\xrightarrowdbl{}X$ about to be described to a compact subset of $X$ (whose preimage in $Y$ will also be compact), so this is a non-issue.

  The ``larger'' space $Y$ will be the following subset of $\bR^3$ (employing the usual Cartesian coordinate system):
  \begin{equation*}
    \begin{aligned}
      Y
      &:=
        Y_{\ell}\cup Y_r\\
      Y_{\ell}
      &=
        \left\{(x,y,z)\in \bR^3\ |\ x\le 0,\ y=z\right\}
        \quad(\text{a half-plane})\\
      Y_{r}
      &=
        \left\{z^2=x\right\}
        \quad(\text{a {\it parabolic cylinder} \cite[\S 5.6, Exercise 6]{alb_solid-an-geom_1949}}). 
    \end{aligned}    
  \end{equation*}
  The branched cover, with fibers of cardinality $\le 3$, is the composition
  \begin{equation*}
    \begin{tikzpicture}[>=stealth,auto,baseline=(current  bounding  box.center)]
      \path[anchor=base] 
      (0,0) node (l) {$Y$}
      +(3,.5) node (u) {$\bR^2\cong \left\{z=0\right\}$}
      +(7,0) node (r) {$X:=\bR^2/\left(\bZ/2\right)$,}
      ;
      \draw[->>] (l) to[bend left=6] node[pos=.5,auto] {$\scriptstyle \text{orthogonal projection along the $z$-axis}$} (u);
      \draw[->>] (u) to[bend left=6] node[pos=.5,auto] {$\scriptstyle \text{quotient map}$} (r);
      \draw[->>] (l) to[bend right=6] node[pos=.5,auto,swap] {$\scriptstyle \pi$} (r);
    \end{tikzpicture}
  \end{equation*}
  with the generator of $\bZ/2$ acting by negating $x$-coordinates: $(x,y,z)\mapsto (-x,y,z)$. Note that the {\it first} of the two maps is certainly not open, but the overall composition is.

  In the language of \cite[p.340]{pt_brnch}, $X$ decomposes into strata $X_k$, $k=1,2,3$, consisting respectively of points with preimage of cardinality exactly $k$:
  \begin{equation*}
    \begin{aligned}
      X_1
      &=
        \left\{\text{image of the origin}\right\},\\
      X_2
      &=
        \left\{\text{image of the $y$-axis}\right\}\setminus X_1, and\\
      X_3
      &=
        X\setminus \left(\widehat{X}_2:=X_1\cup X_2\right).
    \end{aligned}    
  \end{equation*}
  The main feature we are interested in is that {\it every} $X_i$ consists of cluster points of {\it every} $X_j$, $j>i$. We relegate the substance of the example to \Cref{le:23.sheets}.
\end{example}

\begin{lemma}\label{le:23.sheets}
  Let $Y\xrightarrowdbl{\pi}X$ be the branched cover of \Cref{ex:pt_2-3_sheets}.

  The inner product $\braket{f\mid g}$ defined in \cite[proof of Theorem 4.3]{pt_brnch} for functions defined on preimages
  \begin{equation*}
    \pi^{-1}(U)\subseteq Y
    ,\quad
    \text{sufficiently small neighborhood }
    U\ni p_0:=\pi(\text{origin})
  \end{equation*}
  is not, in general, continuous at points on $\pi^{-1}(X_1)$. 
\end{lemma}
\begin{proof}
  For functions $f$ defined on (subsets of) $Y$, write
  \begin{itemize}[wide]
  \item $f_{\ell}$ for the restriction to the ``left-hand'' sheet $Y\cap \{x\le 0\}$;

  \item similarly, $f_r:=f|_{Y\cap \{x\ge 0\}}$;

  \item $f_{ru}$ (the second subscript stands for `up') for the restriction $f|_{Y\cap\{x\ge 0,z\ge 0\}}$ to the ``upper'' sheet of the parabolic cylinder $Y_r$;

  \item and $f_{rd}:=f|_{Y\cap\{x\ge 0,z\le 0\}}$. 
  \end{itemize}
  The same notation applies to preimages of points in $X$: $p_{\ell}$ is the preimage lying on the left-hand sheet, etc.
  
  We step into the proof of \cite[Theorem 4.3]{pt_brnch} at the point where, having defined the restriction $\braket{f\mid g}|_{\widehat{X_2}}$, we seek to extend it continuously across $X_3$ and hence to all of $X$.

  The point $x$ in the paragraph following \cite[(5)]{pt_brnch} will, for us, be the image $p_0\in X$ of the origin in $\bR^3$. Formula \cite[(6)]{pt_brnch} makes it clear that at $p\in X_2$ we have
  \begin{equation}\label{eq:2pts}
    \braket{f\mid g}(p) = \frac 12\left(\overline{f(p_{\ell})}g(p_{\ell}) + \overline{f(p_r)}g(p_r)\right):
  \end{equation}
  the two products are thus evaluated at a point $p_{\ell}$ on the line $\{x=0,\ y=z\}\subset Y_{\ell}$ and $\{x=0,\ z=0\}\subset Y_r$ respectively. On the other hand, for $q\in X_3$ the same \cite[(6)]{pt_brnch} dictates that we set
  \begin{equation}\label{eq:3pts}
    \braket{f\mid g}(q)
    =
    \frac 13\left(\overline{f(q_{\ell})}g(q_{\ell}) + \overline{f(q_{ru})}g(q_{ru}) + \overline{f(q_{rd})}g(q_{rd})\right).
  \end{equation}
  Now, if points $q\in X_3$ (always in the sufficiently small neighborhood $U\ni p_0$ that that portion of the proof focuses on) approach $p\in X_2$, \Cref{eq:3pts} will converge to
  \begin{equation*}
    \frac{\overline{f(p_{\ell})}g(p_{\ell}) + 2\overline{f(p_r)}g(p_r)}{3};
  \end{equation*}
  plainly, this will (in general) differ from \Cref{eq:2pts}. 
\end{proof}

There are other ways to see that the proof of \cite[Theorem 4.3]{pt_brnch} cannot go through as phrased. It will be convenient to regard the embedding $C\left(X\right)\lhook\joinrel\to C\left(Y\right)$ dual to a branched cover $Y\xrightarrowdbl{\pi} X$ of compact Hausdorff spaces as that into the section space $\Gamma(\cA)\cong C\left(Y\right)$ of a $C^*$ bundle $\cA\xrightarrowdbl{}X$. The Hilbert $C\left(X\right)$-module structure on $C\left(Y\right)$ whose existence the theorem asserts amounts to selecting a state
\begin{equation*}
  \cA_x
  \xrightarrow{\quad\varphi_x\quad}
  \bC
\end{equation*}
on each fiber $\cA_x\cong C\left(\pi^{-1}(x\right))$, $x\in X$ in such a manner that the Hilbert-space norms on the finite-dimensional fibers $\cA_x$ are uniformly (in $x$) equivalent to the original norms. One claim implicit in the inductive step in the proof of \cite[Theorem 4.3]{pt_brnch} is the fact that for every point $x\in X$, such states can be selected on small neighborhoods $U\ni x$ so as to take {\it rational} values on the minimal projections of the commutative algebras $\cA_{x'}$, $x'\in U$ (see formulas \cite[(5), (6) and (7)]{pt_brnch}). This rationality constraint will be violated in judiciously chosen examples.

\begin{example}\label{ex:3.sheeted.disk}
  Set
  \begin{equation*}
    X:=\text{unit disk}\{z\in \bC\cong \bR^2\ :\ |z|\le 1\},
  \end{equation*}
  while $Y$ will be a subspace of $\bR^3\cong \bC\times \bR$ with $Y\xrightarrowdbl{\pi}X$ the restriction of the first projection $\bC\times \bR\xrightarrowdbl{}\bC$. We proceed to define $Y$:
  \begin{itemize}[wide]
  \item it contains the two circles union
    \begin{equation*}
      \bS^1_{\pm}
      :\xlongequal{\text{respectively}}
      \bS^1\times \left\{\pm 1\right\}
      \subset
      \bS^1\times \bR;
    \end{equation*}
    
  \item and also the path
    \begin{equation}\label{eq:path.gamma}
      \gamma:=
      \left\{
        \left(e^{\pi i (\theta+1)},\ \theta\right)
        \ |\ \theta\in [-1,1]
      \right\}
    \end{equation}
    starting at $(1,-1)\in \bC\times \bR$, ending at $(1,1)$ and going once counterclockwise around the cylinder $\bS^1[-1,1]$;

  \item and finally, $Y$ is the origin-tipped cone with $\gamma\cup \bS^1_+\cup S^1_-$ as base:
    \begin{equation*}
      Y:=
      \left\{tp\ :\ t\in [0,1],\ p\in \gamma\cup \bS^1_+\cup \bS^1_-\right\}.
    \end{equation*}
  \end{itemize}
  The branched cover $Y\xrightarrow{\pi}X$ has preimages of sizes 1 up to 3:
  \begin{equation*}
    |\pi^{-1}(x)|=
    \begin{cases}
      1&\text{if $x$ is the origin }0\in\bC\cong \bR^2\\
      2&\text{if $x\in [0,1]\cong [0,1]\times \{0\}\subset \bC\times \bR$}\\
      3&\text{otherwise}
    \end{cases}
  \end{equation*}
  For points $x$ with two preimages write $x_{\pm}$ for those preimages, with the `$+$' subscript indicating the higher (larger $z$ coordinate). Similarly, for three-preimage $x$ we write $x_{\pm}$ or $x_0$ for those preimage, with $z$ coordinates ordered as $-<0<+$. 

  A conditional expectation $C\left(Y\right)\xrightarrow{E}C\left(X\right)$ gives a weak$^*$-continuous selection
  \begin{equation}\label{eq:prob.val.map}
    X\ni x
    \xmapsto{\quad}
    \mu_x\in\cat{Prob}(\pi^{-1}(x))\subset \cat{Prob}(Y)\subset C\left(Y\right)^*,
  \end{equation}
  and as
  \begin{equation*}
    x
    =
    \left(e^{\pi i (\theta+1)},\ 0\right)
    ,\quad
    \theta\in [-1,1]
  \end{equation*}
  describes a counter-clockwise revolution around the unit circle $\bS^1$ (image of the path $\gamma$ of \Cref{eq:path.gamma}), $\mu_x(x_-)$ and $\mu_x(x_0)$ must be close to $\mu_{(1,0)}((1,0)_-)$ for small $\theta$, whereas at the other extreme, for $\theta$ close to $2\pi$,
  \begin{equation*}
    \mu_x(x_+)
    \sim
    \mu_{(1,0)}((1,0)_+)
    \sim
    \mu_x(x_0)
  \end{equation*}
  instead (`$\sim$' meaning `close to'). The circle being connected and the values involved rational, they must all be constant with varying $x\in \bS^1$. But then
  \begin{equation*}
    \mu_{(1,0)}((1,0)_{\pm})
    \quad\text{and}\quad
    \mu_x(x_{\pm,0})
  \end{equation*}
  are all equal, which is of course impossible: the two sets are probability distributions on a 2- and 3-element set respectively. The argument furthermore applies in arbitrarily small neighborhoods of the origin $(0,0)\in \bC\times \bR$ because every origin neighborhood includes a scaled version of the entire space $Y$.

  Conclusion: rational-valued expectations $C(\pi^{-1}(U))\xrightarrow{} C(U)$ cannot be defined for neighborhoods $U\ni 0\in X\subset \bC$, no matter how small. 
\end{example}

\begin{remarks}\label{res:blnch_metr}
  \begin{enumerate}[(1),wide]
  \item \cite[Theorem 2.12]{2409.03531v1}, mentioned in the proof of \Cref{th:char.metr}, is very much in the same spirit as the portion of \cite[Th\'eor\`eme 3.3]{blnch} asserting that separable $C_0(X)$-algebras admit continuous fields of states (recalled also in \cite[Remark 3.10]{bg_cx-exp}). The technique employed in proving that result is very similar, hewing more closely to Michael's \cite[Theorem 3.2'']{mich_contsel-1}: rather than citing Michael's theorem directly, the auxiliary \cite[Lemme 3.8]{blnch} rather adapts the proof for maps into {\it Fr\'echet} rather than Banach spaces (i.e. \cite[\S 10]{trev_tvs} locally convex completely metrizable topological vector spaces).
    
    In reference to all of this, the proof of \cite[Lemme 3.8]{blnch} appears to me to contain a gap. The issue is with the assertion that for separable $A$ (a $C^*$-algebra there, but the Banach structure is what matters), the dual space $A^*$ is Fr\'echet under the topology of {\it compact convergence} (as the name suggests, the topology of convergence on compact sets; it is the topology denoted by $\fT_c(A)$ in \cite[\S 21.6]{k_tvs-1} and $\cT_c$ in \cite[\S IV.3.5]{bourb_tvs}, say). That claim cannot hold for {\it any} infinite-dimensional Banach space:
    \begin{itemize}[wide]
    \item The class of compact subsets of a Banach space $E$ is a {\it saturated} cover in the sense of \cite[\S 39.1]{k_tvs-2}, whence the necessary and sufficient criterion \cite[\S 39.4(6)]{k_tvs-2} for the metrizability of the topology $\fT_c(E)$: 
      \begin{equation*}
        \left(
          \exists \text{ compact }
          K_1\subset K_2\subset \cdots
        \right)
        \left(
          \forall \text{ compact }K
        \right)
        \exists n
        \quad:\quad
        K\subseteq K_n.
      \end{equation*}
      Or: $E$ is what is sometimes referred to as {\it hemicompact} (\cite[Problem 17I]{wil_top}, \cite[p.486, Definition]{ar_top}).
      
    \item For {\it first-countable} spaces (\cite[Definition 10.3]{wil_top}: every point has a countable neighborhood basis) hemicompactness is equivalent \cite[p.486, item b) following Definition]{ar_top} to local compactness, 
      
    \item and hence no infinite-dimensional Banach space can be hemicompact \cite[\S 15.7(1)]{k_tvs-1}.
    \end{itemize}
    

  \item The preceding observation notwithstanding, the {\it statement} of \cite[Lemme 3.8]{blnch} appears to be valid. Rather than demand that an appropriate topology on $A^*$ (for a separable continuous unital $C(X)$-algebra $A$) be metrizable, one can simply observe that the LSC map
    \begin{equation*}
      X\ni x
      \xmapsto{\quad}
      \text{state space }\cS(A_x)
      \in 2^{A^*}
    \end{equation*}
    takes values in the metrizable (because $A$ is separable \cite[Theorem V.5.1]{conw_fa}) compact \cite[Theorem V.3.1]{conw_fa} weak$^*$-topologized unit ball of $A^*$, and this suffices \cite[Theorem 1.2]{zbMATH03285029} to ensure the existence of a continuous selection. 

    This is also the type of machinery employed in proving \cite[Theorem 2.12]{2409.03531v1}, which cites \cite[Theorem 3.4]{horv_top-cvx} (see also \cite[Theorem 5]{zbMATH00590976}). This allows for abstract metrizable uniform spaces which nevertheless are not necessarily housed by {\it metrizable} locally convex spaces: the preeminent example is the weak$^*$-topologized unit ball of the dual of a separable Banach space.
        
  \item Concerning the same material, the discussion on \cite[p.155, post Remarque]{blnch} gives a slightly misleading account of why separability is crucial: it refers to \cite[Theorem 3.2'']{mich_contsel-1} as demanding it, but that result does not require separability for the Banach spaces employed (rather \cite[Theorem 3.1'']{mich_contsel-1} does, playing no role in \cite{blnch}). The issue is not so much the separability of the topological vector spaces spaces housing the convex values of the maps as their metrizability.
    
  \item The necessity of the hemicompactness of (the completely regular space) $X$ for $\cat{Cont}(X\to \bR)$ to be first-countable is proven in \cite[Theorem 8]{ar_top}; the result \cite[\S 39.4(6)]{k_tvs-2} appealed to above is a linear version thereof. 
  \end{enumerate}
\end{remarks}

Note also that the proof strategy proposed for \cite[Theorem 4.3]{pt_brnch}, of defining an expectation $C(Y)\xrightarrow{E}C(X)$ recursively over progressively higher strata $X_d$, can fail in other ways than heretofore recorded: even when such expectations do exist, an injudicious choice over a lower stratum $X_d$ might not extend across $X_{d+1}$. 

\begin{example}\label{ex:prod.2.uncount.ord}
  Let $\Omega$ be the smallest uncountable ordinal, give $[0,\Omega]$ (ordinals $\le \Omega$) the order topology (the space of \cite[Example 4.3]{ss_countertop}), and set $X:=[0,\Omega]^2$. We define a branched cover $Y\xrightarrowdbl{\pi}X$ with fiber cardinalities $1$, $2$ and $3$ as
  \begin{equation*}
    \left(Y=X\times\{\bullet,\uparrow,\downarrow\}/\sim\right)
    \xrightarrowdbl[]{\text{first projection }\pi}
    X,
  \end{equation*}
  where the equivalence relation $\sim$ identifying
  \begin{itemize}
  \item all points $(\Omega,\Omega,t)$, $t\in \{\bullet,\uparrow,\downarrow\}$;
  \item every $(\Omega,\alpha,\bullet)$ to $(\Omega,\alpha,\uparrow)$ for $\alpha\in [0,\Omega]$;
  \item and similarly, every $(\alpha,\Omega,\bullet)$ to $(\alpha,\Omega,\downarrow)$ for $\alpha\in [0,\Omega]$.
  \end{itemize}
  The strata are thus 
  \begin{equation*}
    \begin{aligned}
      X_1 &= \{(\Omega,\Omega)\}\\
      X_2 &=
            \{\Omega\}\times [0,\Omega)
            \ \sqcup\ 
            [0,\Omega)\times \{\Omega\}\\
      X_3 &= X\setminus \left(X_1\sqcup X_2\right)
            =[0,\Omega)^2.
    \end{aligned}    
  \end{equation*}
  Note incidentally that all of these are normal (unlike the strata of the examples produced by \Cref{th:pt.non.metr}, as observed in \Cref{le:strt.not.norm}): $X_1$ obviously, $X_2$ simply because it is an ordinal space \cite[Example 39 6.]{ss_countertop}, and $X_3$ by \cite[Corollary 3.3, (c') $\Leftrightarrow$ (f')]{zbMATH00090161}.
  
  Suppose, now, that one has already defined an expectation in the form of a probability-valued map \Cref{eq:prob.val.map} over $X_{\le 2}$ with $\mu_x$ uniform (in this case, this means assigning equal mass $\frac 12$ to the two points constituting the preimage of every $x\in X_2$). I claim that no continuous extension \Cref{eq:prob.val.map} over $X_3$ (and hence all of $X$) is possible. Indeed, writing $x_{\bullet,\updownarrow}$ for the preimages of $x\in X_3$ contained respectively in the sheets $X_3\times \{\circ\}$, $\circ\in \{\bullet,\updownarrow\}$, the functions
  \begin{equation*}
    X_3\cong [0,\Omega)^2
    \xrightarrow[\text{continuous}]{\quad\mu_{x}(x_{\square})\quad}
    (0,1)
    ,\quad
    \square\in\{\bullet,\updownarrow\}
  \end{equation*}
  will stabilize for $x\in [0,\infty)^2$ with both coordinates sufficiently large: this is no harder to prove than for $[0,\Omega)$ (in place of $[0,\Omega)^2$), as in \cite[Example 43 12.]{ss_countertop}, say. On the other hand though, 
  \begin{equation*}
    \begin{tikzpicture}[>=stealth,auto,baseline=(current  bounding  box.center)]
      \path[anchor=base] 
      (0,0) node (l) {$(\mu_x(x_{\uparrow}),\ \mu_x(x_{\bullet}),\ \mu_x(x_{\downarrow}))$}
      +(7,.5) node (u) {$\left(\frac 12,\ *,\ \frac 12-*\right)$}
      +(7,-.5) node (d) {$\left(\frac 12-*,\ *,\ \frac 12\right)$}
      ;
      \draw[->] (l) to[bend left=6] node[pos=.6,auto] {$\scriptstyle x\to(\alpha,\Omega)\text{ along }(\alpha,?)$} (u);
      \draw[->] (l) to[bend right=6] node[pos=.6,auto,swap] {$\scriptstyle x\to(\Omega,\alpha)\text{ along }(?,\alpha)$} (d);
    \end{tikzpicture}
  \end{equation*}
  This (given the assumed strict positivity of the $\mu_x$) contradicts said stabilization.
\end{example}

The slight awkwardness in the phrasing of \Cref{th:pt.non.metr} (referring, as it does, to truncated power-sets $Z_{[n]}$ rather than to spaces $Z$ directly), in conjunction with the positive result \cite[Theorem 2.12]{2409.03531v1} under metrizability assumptions on the base space $X$ of a subhomogeneous $C^*$-bundle, suggests further questions. One might examine, for instance, the class of compact Hausdorff spaces $X$ for which {\it all} subhomogeneous continuous $C(X)$-algebras admit (finite-type) conditional expectations. That these need not {\it all} be metrizable is not difficult to see. Some vocabulary will help streamline the discussion. 

\begin{definition}\label{def:brnching}
  A compact Hausdorff space $X$ is
  \begin{enumerate}[(1),wide]
  \item\label{item:def:brnching:cls} {\it universally branching} if the embedding $C(X)\lhook\joinrel\to C(Y)$ dual to any branched cover $Y\xrightarrowdbl{}X$ admits a finite-index expectation.

  \item\label{item:def:brnching:nc} {\it universally nc-branching} (for `non-commutative') if every subhomogeneous continuous unital $C(X)$-algebra with non-zero fibers admits a finite-index expectation. 
  \end{enumerate}
  We might occasionally use the terms in verb form: a space {\it (nc-)branches universally}.   
\end{definition}

Part \Cref{item:pr:extr.disc:subhom} of the following result says that, for appropriately well-behaved $X$, not only do all continuous subhomogeneous $C(X)$-algebras admit finite-index expectations, but in fact do so as ``tightly'' as possible

\begin{proposition}\label{pr:extr.disc}
  Let $X$ be an {\it extremally disconnected} compact Hausdorff space, i.e. \cite[\S I.4, p.32]{ss_countertop} one whose open set closures are open.

  \begin{enumerate}[(1),wide]
  \item\label{item:pr:extr.disc:gen} Every continuous unital $C(X)$-algebra with non-zero fibers admits an expectation. 
    
  \item\label{item:pr:extr.disc:subhom} Every continuous subhomogeneous unital $C(X)$-algebra $C(X)\xrightarrow{\iota}A$ with non-zero fibers admits an expectation $C(X)\xrightarrow{E}A$ with K-constant \Cref{eq:k.const} as small as possible \cite[Theorem 1.4]{bg_cx-exp}, i.e. equal to the {\it rank} \cite[Remark 2.6]{bg_cx-exp}
    \begin{equation*}
      r(A)=r(\iota):=
      \sup_{x\in X}\left(\sum\text{dimensions of irreducible $A_x$-representations}\right). 
    \end{equation*}
    In particular, $X$ is universally nc-branching. 
  \end{enumerate}
\end{proposition}
\begin{proof}
  \begin{enumerate}[(1),wide]
  \item The $C^*$ bundle $\cA\xrightarrowdbl{\pi}X$ associated to $A$ induces a surjection
    \begin{equation*}
      \cS(\cA)\xrightarrowdbl{\pi^*}X
      ,\quad
      \cS(\cA):=\bigcup_{x\in X}\left(\text{state space }\cS(\cA_x)\right)
      \subset \text{state space }\cS(A=\Gamma(\cA)),
    \end{equation*}
    where the inclusion is that induced by the surjections $A\xrightarrowdbl{}A_x=A/I_x\cdot A$ ($\cS(\cA)$ is what \cite[\S 3.1]{blnch} denotes by $\cS_X(A)$). Having topologized $\cS(\cA)$ with weak$^*$ subspace topology inherited from that embedding, $\pi^*$ is a continuous map of compact Hausdorff spaces by \cite[Proposition 15.4]{gierz_bdls}. $X$ being by assumption extremally disconnected and hence \cite[Theorem 2.5]{gls_proj-top-sp} {\it projective} in the category of compact Hausdorff spaces (i.e. \cite[\S V.4]{mcl_2e} {\it epimorphisms} with codomain $X$ split), that map has a right inverse:
    \begin{equation*}
      \exists
      X\lhook\joinrel\xrightarrow[\text{continuous}]{\quad\iota\quad}
      \cS(\cA)
      ,\quad
      \pi^*\circ\iota=\id_X. 
    \end{equation*}
    The target expectation is
    \begin{equation*}
      A=\Gamma(\cA)\ni s
      \xmapsto{\quad E\quad}
      \bigg(
      X\ni x
      \xmapsto{\quad}
      \iota(x)(s(x))
      \bigg)
      \in C(X).
    \end{equation*}
       
  \item The above proof applies, with the space
    \begin{equation}\label{eq:s.fixed.k}
      \cS_{K=r(\cA)}(\cA)
      :=
      \bigcup_{x\in X}
      \left\{\varphi\in \cS(\cA_x)\ |\ r(\cA)\cdot \varphi-\id\ge 0\right\}
      \subset \cS(\cA)
    \end{equation}
    in place of the larger $\cS(\cA)$ used in the proof of part \Cref{item:pr:extr.disc:gen}. The composition
    \begin{equation*}
      \cS_{K=r(\cA)}(\cA)
      \lhook\joinrel\xrightarrow{\quad}
      \cS(\cA)
      \xrightarrowdbl{\quad\pi^*\quad}
      X
    \end{equation*}
    is still a surjection, given that \cite[Lemma 3.2]{bg_cx-exp} the individual state spaces
    \begin{equation*}
      \left\{\varphi\in \cS(\cA_x)\ |\ r(\cA_x)\cdot \varphi-\id\ge 0\right\}
      \subseteq
      \left\{\varphi\in \cS(\cA_x)\ |\ r(\cA)\cdot \varphi-\id\ge 0\right\}
    \end{equation*}
    are non-empty for all $x\in X$. 
  \end{enumerate}
\end{proof}

\Cref{pr:extr.disc}\Cref{item:pr:extr.disc:subhom} is in a sense at the opposite end of a spectrum from the analogous result \cite[Theorem 2.12]{2409.03531v1} for metrizable $X$: there are no infinite extremally disconnected dyadic spaces \cite[Theorem 4]{zbMATH03199381}, let alone metrizable ones (compact metric spaces being dyadic \cite[Theorem 8.3.6]{sem_ban-sp-cont}). 

We will see later (\Cref{cor:ord.nc.brnch}) that orderable spaces (brought up in \Cref{res:ordrbl}\Cref{item:res:ordrbl:mich.ord} in connection with \Cref{ex:lex.sq}) behave well branching-wise. In preparation for that, we record in \Cref{le:ext.same.k} a slight generalization of \cite[Proposition 3.4]{bg_cx-exp}. {\it Homogeneous} bundles \cite[Definition 2.2]{dupre_hilbund-1} are those whose fibers have constant finite dimension.


\begin{lemma}\label{le:ext.same.k}
  Let $\cA\xrightarrowdbl{\pi}X$ be a continuous homogeneous $C^*$ bundle over a paracompact (Hausdorff) space and $F\subseteq X$ a closed subset.

  A finite-index expectation $\Gamma\left(\cA|_F\right)\xrightarrow{E|_F}C(Y)$ extends to one such over all of $X$, with the same K-constant. 
\end{lemma}
\begin{proof}
  Forming convex combinations cannot increase K-constants past the supremum among the members entering the combination:
  \begin{equation*}
    K\left(\sum c_i E_i\right)\le \sup_i K(E_i)
    \quad\text{for}\quad
    c_i\ge 0,\ \sum c_i=1. 
  \end{equation*}
  The usual technique of defining the desired extension $E$ locally on open sets $U_i$ and then passing to $\sum \varphi_i E_i$ for a {\it partition of unity} \cite[Theorem 5.1.9]{eng_top_1989} $(\varphi_i)_i$ subordinate to the cover $(U_i)_i$ reduces the problem to its local analogue. 
  
  Because by \cite[Theorem 5.1]{hk_shv-bdl} bases of fibers $\cA_x$ extend locally around $x$ to linearly independent \cite[p.231, Remarque]{dd} sections, we may as well (having shrunken $X$ appropriately) assume the bundle trivial: $\cA\cong X\times A$ for some finite-dimensional $C^*$-algebra $A$. The desired expectation $E$ is now nothing but a continuous map
  \begin{equation*}
    X
    \xrightarrow{\quad}
    \cS_{K=K(E|_F)}(A)
    :\xlongequal{\quad\text{cf. \Cref{eq:s.fixed.k}}\quad}
    \left\{\varphi\in \cS(A)\ |\ K(\varphi)=K(E|_F)\right\},
  \end{equation*}
  and the extensibility of such a map from the closed subset $Y\subseteq X$ to all of $X$ follows from the {\it Tietze extension theorem} \cite[Theorem 15.8]{wil_top} (given that the set $\cS_{K=K(E|_F)}(A)$ is compact convex in some Euclidean space). 
\end{proof}


As \Cref{ex:pt_2-3_sheets} pointed out, one complicating issue, when defining/extending expectations $\Gamma_b(\cA)\to C_b(X)$ is the fact that points might have multiple strata interacting in arbitrarily small neighborhoods. The following notion is meant formalizes one way in which this does {\it not} occur.

\begin{definition}\label{def:tame.strat}
  A subhomogeneous (F) Banach bundle $\cA\xrightarrowdbl{}X$ is {\it tamely stratified} if for every $d\in \bZ_{\ge 0}$ the open subspace $X_d\subseteq X_{\le d}$ is a disjoint union
  \begin{equation}\label{eq:xd.clopen.part}
    X_d = \coprod_{i\in I}U_i
    ,\quad
    U_i\text{ clopen in }X_{\le d}
  \end{equation}
  so that
  \begin{equation*}
    \left(\text{boundary of $U_i$ in $X_{\le d}$}=:\right)
    \partial_{X_{\le d}} U_i
    \ni x
    \xmapsto{\quad}
    \dim \cA_x
    \in \bZ_{\ge 0}
    \text{ is locally constant}.
  \end{equation*}
\end{definition}

\begin{theorem}\label{th:tame.strat}
  A tamely stratified continuous unital subhomogeneous $C^*$-bundle with non-zero fibers over a paracompact Hausdorff space is a non-commutative branched cover. 
\end{theorem}
\begin{proof}
  Let $\cA\xrightarrowdbl{\pi}X$ be the bundle in question. The sought-after expectation
  \begin{equation*}
    A:=\Gamma_b(\cA)
    \xrightarrow{\quad E\quad}
    C_b(X)
  \end{equation*}
  amounts to a weak$^*$-continuous map
  \begin{equation*}
    X\ni x
    \xmapsto{\quad}
    E_x\in \cS(\cA_x)\le \cS(\Gamma(\cA))
  \end{equation*}
  with $K(E_x)$ bounded in $x$. We construct that map recursively over the strata $X_{\le d}$ of \Cref{eq:strata} for increasing $d$. There will furthermore be no need to keep track of the finite-index condition: the construction will make it clear that the K-constants $K(E_x)$ (defined as in \Cref{eq:k.const}) of the states $\cA_x\xrightarrow{E_x}\bC$ that constitute the expectation are bounded uniformly (in $x$).
  
  To initiate the recursion, note \cite[Proposition 16.4 and Theorem 18.5]{gierz_bdls} that $\cA$ is locally trivial over the ``lowest'' stratum
  \begin{equation*}
    X_{\le d_0} = X_{d_0} = \left\{x\in X\ |\ \dim \cA_x=d_0:=\min_{y}\dim \cA_y\right\}
  \end{equation*}
  and take for $E_x\in \cS(\cA_x)$, $x\in X_{d_0}$ the minimal-K-constant state on that $C^*$-algebra (that state is canonically attached \cite[Lemma 3.2]{bg_cx-exp} to the $C^*$-algebra in question, hence its continuous variation in locally trivial $C^*$ bundles). 

  Assume next that we have defined $E_x$ for $x\in X_{\le d}$ and seek to extend the definition over $X_{\le d'}$ for the next higher fiber dimension $d'>d$. We may as well assume $X=X_{\le d'}$, by way of simplifying the notation.

  Let
  \begin{equation}\label{eq:xd.clopen.part}
    X_{d'} = \coprod_{i\in I}U_i
    ,\quad
    U_i\text{ clopen in }X
  \end{equation}
  be the decomposition afforded by \Cref{def:tame.strat} and the tameness assumption. Because $E_x$ is assumed defined over the boundaries
  \begin{equation*}
    \partial U_i\subseteq X_{\le d},
  \end{equation*}
  it will suffice to extend that definition continuously over every closure $\overline{U_i}=U_i\sqcup \partial U_i$ individually: the restrictions $E|_{\overline{U_i}}$ thus obtained will then glue to a single expectation $E$. We assume, in short, that $X=\overline{U_i}$ for one single $i$ (so we may as well write $U:=U_i$), $E|_{F:=\partial U}$ is already in place, and what is needed is its extension across $U$.

  \begin{enumerate}[(I),wide]
  \item\label{item:frwrd.loc} {\bf : Moving inward into $U$ from its boundary $F$, locally on the latter.} Fix for the moment an $x_0\in F$, and a closed neighborhood $G_{x_0}\ni x_0$ sufficiently small to ensure that the restriction
    \begin{equation*}
      \cA|_{G_{x_0}}
      \cong G_{x_0}\times A
      ,\quad\text{$A$ a finite-dimensional $C^*$-algebra}
    \end{equation*}
    is trivialized by sections $(s_{x_0,j})_j$ of $\cA|_{G_{x_0}}$ providing $C^*$ embeddings
    \begin{equation}\label{eq:nearby.emb}
      \cA_{x_0}
      \ni
      \sum c_j s_{x_0,j}(x_0)
      \xmapsto{\quad\iota_x\quad}
      \sum c_j s_{x_0,j}(x)
      \in
      \cA_x
      ,\quad x\in G_{x_0}
    \end{equation}
    (that such sections exist is ensured by \cite[Theorem 3.1]{fell_struct}).
    
    The slices $E_x$, already defined for 
    \begin{equation*}
      x\in F_{x_0}:=F\cap G_{x_0},
    \end{equation*}
    constitute a continuous map
    \begin{equation*}
      F_{x_0}
      \xrightarrow{\quad}
      \text{state space }\cS(A),
    \end{equation*}
    extensible continuously to $G_{x_0}$ by Tietze \cite[Theorem 15.8]{wil_top} to provide faithful states $E_x|_{\iota_x(\cA_{x_0})}$ varying $x$-continuously. 
    
    Writing $\mathrm{tr}$ for the canonical minimal-K-constant (tracial \cite[Lemma 3.2]{bg_cx-exp}) state on whatever finite-dimensional $C^*$-algebra the notation applies to, the original expectation $E|_{F_{x_0}}$ extends to all of $G_{x_0}$ by 
    \begin{equation*}
      E_x := \mathrm{tr}(t_x\cdot \bullet)
      ,\quad
      x\in G_{x_0}
    \end{equation*}    
    with $t_x\in \cA_x$ the (unique) element therein with
    \begin{equation*}
      \mathrm{tr}(t_x\cdot\bullet) = E_x\text{ on }\iota_x(\cA_{x_0}),
    \end{equation*}
    closest to $\iota_x(\cA_{x_0})\le \cA_x$ with respect to the (norm induced by the) inner product $\braket{a\mid b}:=\mathrm{tr}(a^*b)$ on $\cA_x$.

  \item\label{item:frwrd.glob} {\bf Globalizing \Cref{item:frwrd.loc} over all of $F$.} Cover $F$ with the $G_{x}$, $x\in F$ of step \Cref{item:frwrd.loc}, fix a closed subset $G\subseteq X$ with 
    \begin{equation*}
      G\subseteq \bigcup_{x\in F}G_x^{\circ}      
      \quad\text{and}\quad
      F\subseteq G^{\circ}
    \end{equation*}
    with `$^\circ$' denoting interiors (possible by paracompactness and hence \cite[Theorem 20.10]{wil_top} normality), and glue the restrictions $E|_{G\cap G_x}$ provided by that preceding step by means of a partition of unity on $G$ subordinate to the open cover
    \begin{equation*}
      G=\bigcup_{x\in F}\left(G\cap G^{\circ}_x\right).
    \end{equation*}
    
  \item\label{item:all} {\bf : Extension to all of $X$.} We now have $E|_G$, where $G\subseteq X$ is closed and contains the ``degenerate'' stratum $X_{\le d}\subset X$ in its interior. This means that on the closed subspace $\overline{X\setminus G}\subseteq X$ the bundle $\cA$ is homogeneous, and we can extend $E$ from its boundary $\partial(X\setminus G)$ with no K-constant increase by \Cref{le:ext.same.k}.
  \end{enumerate}  
\end{proof}

\begin{corollary}\label{cor:ord.nc.brnch}
  Totally orderable compact Hausdorff spaces are universally nc-branching in the sense of \Cref{def:brnching}\Cref{item:def:brnching:nc}. 
\end{corollary}
\begin{proof}
  This is a consequence of \Cref{th:tame.strat}, given that over totally orderable compact Hausdorff spaces {\it all} subhomogeneous continuous unital $C^*$ bundles with non-zero fibers are tamely stratified in the sense of \Cref{def:tame.strat}.

  Indeed, the closed strata $X_{\le d}$ are themselves orderable, and the open subsets $X_d\le X_{\le d}$ thereof decompose as disjoint unions of open intervals: the $U_i$ of \Cref{eq:xd.clopen.part}. The boundaries $\partial U_i$ will each consist of $\le 2$ interval endpoints, hence tame stratification.
\end{proof}

\begin{remarks}\label{res:emb.germ}
  \begin{enumerate}[(1),wide]
  \item   The embeddings \Cref{eq:nearby.emb} of the fibers of a $C^*$ bundle into their nearby neighbors can be exploited to produce a number of pathological examples. If, for
    \begin{equation*}
      X\ni x,x'
      \text{ arbitrarily close to }
      x_0\in X
    \end{equation*}
    the embeddings
    \begin{equation*}
      \cA_{x_0}\lhook\joinrel\xrightarrow{\iota^{\bullet}} \cA_{x^{\bullet}}
      ,\quad
      \bullet\in\left\{\text{blank},\ '\right\}
    \end{equation*}
    attached to section prolongations locally around $x_0$ are such that no tracial state on $\cA_{x_0}$ is a restriction of tracial states on both $\cA_{x^\bullet}$ along $\iota^{\bullet}$, then the bundle is not embeddable into a homogeneous one. Indeed, consider a homogeneous $\cB\xrightarrowdbl{}X$, trivial on a small neighborhood $U\ni x_0\in X$:
    \begin{equation*}
      \cB|_U\cong U\times B
      ,\quad
      B\text{ a finite-dimensional $C^*$-algebra}.
    \end{equation*}
    An embedding $\cA\le \cB$ will then produce a tracial state on $\cA_{x_0}$ by restricting any fixed tracial state on $B\cong \cB_x$, $x\in U$ first to $\cA_x\le \cB_x$, and then along embeddings $\cA_{x_0}\le \cA_x$ analogous as in \Cref{eq:nearby.emb}. 

    \cite[Example 2.2]{2409.03531v1} is a concrete instance of this:
    \begin{itemize}[wide]
    \item the base space $X$ is $[-1,1]$ (so presumably as well-behaved as one might hope for);

    \item the bundle is trivial, isomorphic to $X_0\times M_3$ over $X_0:=X\setminus \{0\}$;

    \item and has 2-dimensional 0-fiber $\cA_0\cong \bC^2$;

    \item while both
      \begin{equation}\label{eq:2conflicting}
        \bC^2\ni (a,b)
        \xmapsto{\quad}
        \mathrm{diag}(a,a,b)
        \quad\text{and}\quad
        \mathrm{diag}(a,b,b)
        \in M_3
      \end{equation}
      are achievable as embeddings $\cA_0\le \cA_x$ for $x\in X$ arbitrarily close to $0$. 
    \end{itemize}
    Naturally, no (automatically tracial) state on $\cA_0\cong \bC^2$ is simultaneously the restriction of the unique tracial state on $M_3$ along both \Cref{eq:2conflicting}.

  \item  The problem of embedding subhomogeneous bundles into homogeneous ones of some interest in \cite{bg_cx-exp}, where \cite[Example 3.6]{bg_cx-exp} (also \cite[Example 3.5]{zbMATH05172034}) has this same feature. There, where the base space $X$ is the one-point compactification of the disjoint union $\bigcup_{n}\bC\bP^n$ of {\it complex projective spaces} \cite[\S 14.3]{ms-cc}, the main issue appears to be that the {\it type} \cite[Definition 3.5.7]{hus_fib} of the bundle over the individual $\bC\bP^n$ is not bounded: trivializations of the bundle are only achievable by increasingly large open covers
    \begin{equation*}
      \bC\bP^n = \bigcup_{k=1}^{\ell_n}U_k
      ,\quad
      \ell_n\xrightarrow[n]{\quad}\infty.
    \end{equation*}
    By contrast, the aforementioned \cite[Example 2.2]{2409.03531v1} has finite-type restrictions over loci where the fibers have constant dimension: the bundle is trivial over $[-1,1]\setminus \{0\}$, where the fiber is $M_3$.
  \end{enumerate}
\end{remarks}




\addcontentsline{toc}{section}{References}

\Addresses

\end{document}